\documentclass[11pt]{amsart}
\usepackage{amsmath,amssymb,latexsym,cite}
\usepackage{color,graphicx, tikz, mathtools,microtype,soul,enumitem}
\usepackage[colorlinks=true,urlcolor=blue,citecolor=red,linkcolor=blue,linktocpage,pdfpagelabels,bookmarksnumbered,bookmarksopen]{hyperref}
\usepackage[english]{babel}
\usepackage[hyperpageref]{backref}
\usepackage[left=3.2cm,right=3.2cm,top=2.9cm,bottom=2.9cm]{geometry}
\usepackage{color,graphicx}

\numberwithin{equation}{section}

\newtheorem{thm}{Theorem}[section]
\theoremstyle{plain}
\newtheorem{lem}[thm]{Lemma}
\theoremstyle{plain}

\theoremstyle{plain}

\theoremstyle{definition}
\newtheorem{rem}[thm]{Remark}
\newtheorem{ex}[thm]{Example}

\newcommand{\N}{{\mathbb N}}

\newcommand{\R}{{\mathbb R}}
\newcommand{\eps}{\varepsilon}

\renewcommand{\le}{\leqslant}
\renewcommand{\ge}{\geqslant}

\def\XXint#1#2#3{{\setbox0=\hbox{$#1{#2#3}{\int}$ }
\vcenter{\hbox{$#2#3$ }}\kern-.57\wd0}}

\DeclareMathOperator*{\esssup}{ess\, sup}

\newcommand{\restr}[2]{\left.#1\right|_{#2}}

\newenvironment{enumroman}{\begin{enumerate}

}{\end{enumerate}}

\title[A partial differential inclusion via nonsmooth Morse theory]{Three solutions for a Neumann partial differential inclusion via nonsmooth Morse theory}

\author[F.\ Colasuonno]{Francesca Colasuonno}
\author[A.\ Iannizzotto]{Antonio Iannizzotto}
\author[D.\ Mugnai]{Dimitri Mugnai}

\address{Dipartimento di Scienze Matematiche ``G.L.\ Lagrange''
\newline\indent
Politecnico di Torino
\newline\indent
Corso Duca degli Abruzzi 24, 10129 Torino, Italy}
\email{francesca.colasuonno@polito.it}

\address{Dipartimento di Matematica e Informatica
\newline\indent
Universit\`a degli Studi di Cagliari
\newline\indent
Viale L. Merello 92, 09123 Cagliari, Italy}
\email{antonio.iannizzotto@unica.it}

\address{Dipartimento di Matematica e Informatica
\newline\indent
Universit\`a degli Studi di Perugia
\newline\indent
Via Vanvitelli 1, 06123 Perugia, Italy}
\email{dimitri.mugnai@unipg.it}


\begin{document}

\begin{abstract}
We study a partial differential inclusion, driven by the $p$-Laplacian operator, involving a $p$-superlinear nonsmooth potential, and subject to Neumann boundary conditions. By means of nonsmooth critical point theory, we prove the existence of at least two constant sign solutions (one positive, the other negative). Then, by applying the nonsmooth Morse relation, we find a third non-zero solution.
\end{abstract}

\maketitle

\section{Introduction}\label{sec1}

In the present paper we deal with the following partial differential inclusion, coupled with homogeneous Neumann boundary conditions:
\begin{equation} \label{prob}
\begin{cases}
-\Delta_p u\in\partial j(x,u) &\text{in $\Omega$,} \\
\displaystyle\frac{\partial u}{\partial\nu}=0 &\text{on $\partial\Omega$.}
\end{cases}
\end{equation}
Here and in what follows, $\Omega\subset\R^N$ ($N\geq 2$) is a bounded domain with $C^2$ boundary $\partial\Omega$, $p>1$ and the $p$-Laplacian operator is defined as $\Delta_p u={\rm div}(|\nabla u|^{p-2}\nabla u)$, while $\partial u/\partial\nu(x)$ denotes the outward normal derivative of $u$ at $x\in\partial\Omega$. Finally, $\partial j(x,s)$ denotes the Clarke generalized subdifferential (with respect to $s$) of a potential $j:\Omega\times\R\to\R$ which is assumed to be measurable in $x$ and locally Lipschitz continuous in $s$.

Problems of the type \eqref{prob} have been studied in a variational framework since the pioneering work of Chang \cite{C1}, dealing with partial differential equations involving a discontinuous reaction term. The best fitting theoretical framework is the nonsmooth critical point theory for locally Lipschitz continuous functionals developed by Clarke \cite{C4}. Such approach led to many results, for instance those contained in the monographs of Carl, Le $\&$ Motreanu \cite{CLM} and of Gasi\'nski $\&$ Papageorgiou \cite{GP1} and the papers of Averna, Marano \& Motreanu \cite{AMM}, Barletta $\&$ Papageorgiou \cite{BP}, Carl \& Motreanu \cite{CM}, Iannizzotto, Marano $\&$ Motreanu \cite{IMM}, Iannizzotto $\&$ Papageorgiou \cite{IP}, Kyritsi $\&$ Papageorgiou \cite{KP}. In all the mentioned works, nonsmooth critical point theory is employed together with some additional tool, such as sub- and supersolutions or spectral theory.

In most results, the potential is either $p$-sublinear or asymptotically $p$-linear at infinity, while here we shall consider a function $j(x,\cdot)$ which is $p$-{\rm superlinear} at infinity. Such a study was developed for the $p$-Laplacian by Bartsch $\&$ Liu \cite{BL}, Degiovanni $\&$ Lancelotti \cite{DL}, Perera \cite{P1} (for a Dirichlet problem) and Aizicovici, Papageorgiou $\&$ Staicu \cite{APS}, Binding, Dr\'abek $\&$ Huang \cite{BDH} (for a Neumann problem). A typical difficulty in the $p$-superlinear case is that one loses control on the asymptotic behavior of the energy functional related to the problem, which suffers as a consequence of a lack of compactness of critical sequences. An usual way to overcome such difficulty is to require some version of the Ambrosetti-Rabinowitz condition (shortly, $(AR)$).

In the $C^1$ framework, Wang \cite{W} proved the existence of three non-zero solutions for a Dirichlet problem driven by the Laplacian operator ($p=2$), involving a reaction term which is superlinear at infinity and satisfies $(AR)$ (see also Mugnai \cite{M}, \cite{M1}, Rabinowitz, Su $\&$ Wang \cite{RSW}, and, for the non smooth case, Magrone, Mugnai $\&$ Servadei \cite{MMS}). The author employed a technique based on Morse theory, mainly on the computation of the critical groups of the energy functional at zero and at infinity. Morse theory is indeed a powerful tool in producing multiplicity results, both for semilinear and for quasilinear problems, as it describes the local behavior of a functional around its critical points, as well as its asymptotic behavior,  through simple algebraic relations. For an introduction to Morse theory, we refer to the monographs by Ambrosetti $\&$ Malchiodi \cite{AM}, Chang \cite{C2}, Motreanu, Motreanu $\&$ Papageorgiou \cite{MMP}, Perera, Agarwal $\&$ O'Regan \cite{PAO}, and to the survey paper of Bartsch, Szulkin \& Willem \cite{BSW}. Interesting applications to several types of $p$-Laplacian problems can be found in the papers of Degiovanni, Lancelotti \& Perera \cite{DLP}, Liu \& Li \cite{LL2}, Liu \& Liu \cite{LL1}, Perera \cite{P}, and Perera \& Sim \cite{PS1}.

The basic notions and results of Morse theory (namely the so-called Morse relation) have been extended to nonsmooth functionals by Corvellec \cite{C}. Exploiting such an extension, we aim at proving the existence of at least three non-zero solutions for problem \eqref{prob} if the potential is $p$-superlinear at infinity and satisfies a mild version of $(AR)$, thus extending to a wider framework the ideas of \cite{W}. In doing so, we shall need to adapt some techniques of Morse theory to the nonsmooth setting (see mainly Sections \ref{sec3} and \ref{sec5} below), which we believe can be useful also for further results. As far as we know, this is the first application of Morse-theoretical ideas in the field of partial differential inclusions.
\medskip

\noindent
Our precise assumptions on the potential $j$ are the following:
\begin{itemize}[leftmargin=0.5cm]
\item[{\bf H}] $j:\Omega\times\R\to\R$ is a Carath\'eodory function such that $j(x,\cdot)$ is locally Lipschitz continuous for a.a. $x\in\Omega$, and $j(\cdot,0)\in L^1(\Omega)$; moreover, it satisfies
\begin{enumroman}
\item\label{h1} there exist $a_0>0$ and $r\in(p,p^*)$ such that
$|\xi|\leq a_0(1+|s|^{r-1})$ for a.a. $x\in\Omega$ and all $s\in\R$, $\xi\in\partial j(x,s)$, where $p^*$ denotes the Sobolev critical exponent
\[p^*=\begin{cases}
Np/(N-p) &\mbox{if } p<N,\\
\infty &\mbox{if }p\ge N;\end{cases}\]
\item\label{h2} $\displaystyle\lim_{|s|\to \infty}\frac{j(x,s)}{|s|^p}=\infty$ uniformly for a.a. $x\in\Omega$;
\item\label{h3} there exists  $q\in\Big((r-p)\max\big\{\frac{N}{p},1\big\},p^*\Big)$ such that  
\[
\liminf_{|s|\to \infty}\min_{\xi\in\partial j(x,s)}\frac{\xi s-p j(x,s)}{|s|^q}>0 \mbox{ uniformly for a.a. $x\in\Omega$};
\]
\item\label{h4} there exists $\delta_0>0$ such that  $j(x,s)\leq j(x,0)$ for a.a. $x\in\Omega$ and all $|s|\leq\delta_0$;
\item\label{h5} there exists $c_0>0$ such that  $\xi s\geq -c_0|s|^p$ for a.a. $x\in\Omega$ and all $s\in\R$, $\xi\in\partial j(x,s)$.
\end{enumroman}
\end{itemize}
Our main result is stated as follows:

\begin{thm}\label{three}
If hypotheses ${\bf H}$ hold, then problem \eqref{prob} admits at least three smooth non-zero solutions $u_+,u_-,\tilde u\in C^1(\overline\Omega)$ such that $u_+(x)>0>u_-(x)$ in $\overline\Omega$.
\end{thm}

We present here two examples of potentials satisfying hypotheses ${\bf H}$:

\begin{ex}\label{ex1}
Let $p<\sigma<p^*$ and set for all $s\in\R$
\[j_1(s)=\begin{cases}
0 \quad&\mbox{if } s=0,\\
\displaystyle \frac{|s|^\sigma}{\sigma}\ln|s| \quad&\mbox{if } 0<|s|<1,\\
\displaystyle \frac{|s|^p}{p}\ln|s| \quad&\mbox{if } |s|\geq 1,
\end{cases}\]
and
\[j_2(s)=\frac{|s|^\sigma}{\sigma}-\frac{|s|^p}{p}.\]
Then, both potentials $j_1,j_2:\R\to\R$ satisfy hypotheses ${\bf H}$ (notice that $j_2$ is of class $C^1$ in $\R$, while $j_1$ is only Lipschitz continuous near $s=1$).
\end{ex}

We remark that our three solutions theorem differs from most results of this type available in the literature. Indeed, while usually one detects two local minimizers of the energy functional and then, by applying the mountain pass theorem, the existence of a third critical point is guaranteed (see for instance \cite{IP}), here we find two critical points of mountain pass type and then a third critical point (of undetermined nature) is found via Morse theory. We also emphasize the fact that we don't require the usual $(AR)$, but a weaker condition, see {\bf H}\ref{h3}.
\smallskip

The paper has the following structure. In Section \ref{sec2} we recall some basic features of nonsmooth critical point theory and prove a nonsmooth implicit function lemma (see Lemma \ref{ift}). Section \ref{sec3} is devoted to a presentation of the nonsmooth Morse theory. In Section \ref{sec4}, we use the nonsmooth mountain pass theorem to prove that \eqref{prob} admits two constant sign smooth solutions, one positive and the other negative (see Theorem \ref{two}). Finally, in Section~\ref{sec5}, after computing the critical groups of the energy fuctional at zero, at such constant sign solutions, and at infinity, we deduce the existence of a third smooth nontrivial solution (of undetermined sign) and conclude the proof of Theorem \ref{three}.

\section{Preliminaries I: nonsmooth critical point theory}\label{sec2}

First, we introduce some notation. Throughout the paper, $(X,\|\cdot\|)$ denotes a reflexive Banach space, $(X^*,\|\cdot\|_*)$ its topological dual and $\langle\cdot,\cdot\rangle$ the duality between $X^*$ and $X$. $B_\rho(u)$, $\overline B_\rho(u)$ and $\partial B_\rho(u)$ will denote the open and closed balls and the sphere in $X$ centered at $u\in X$ with radius $\rho>0$, respectively. For all $\varphi:X\to\R$ and all $c\in\R$ we set
\[\varphi^c=\{u\in X \ : \ \varphi(u)<c\}, \quad \overline\varphi^c=\{u\in X \ : \ \varphi(u)\leq c\}\]
(we also set $\overline\varphi^\infty=X$). Finally, when estimates are considered, we shall denote by $c>0$ positive constants, which are allowed to vary from line to line.

We recall some basic notions and results from nonsmooth critical point theory, based on the ideas of Clarke \cite{C4}, referring to Gasi\'nski $\&$ Papageorgiou \cite{GP1} for details. A functional $\varphi:X\to\R$ is said to be {\em locally Lipschitz continuous} if for every $u\in X$ there exist a neighborhood $U$ of $u$ and $L>0$ such that 
\[|\varphi(v)-\varphi(w)|\leq L\|v-w\| \ \mbox{for all $v,w\in U$.}\]
From now on we assume $\varphi$ to be locally Lipschitz continuous.
The {\em generalized directional derivative} of $\varphi$ at $u$ along $v\in X$ is
\[\varphi^\circ(u;v)=\limsup_{w\to u\above 0pt t\to 0^+}\frac{\varphi(w+tv)-\varphi(w)}{t}.\]
The {\em generalized subdifferential} of $\varphi$ at $u$ is the set
\[\partial\varphi(u)=\left\{u^*\in X^*:\, \langle u^*,v\rangle\leq\varphi^\circ(u;v) \ \mbox{for all $v\in X$}\right\}.\]
The following Lemmas display some basic properties of the notions introduced above, see \cite[Propositions 1.3.7-1.3.12]{GP1}:

\begin{lem}\label{gd}
If $\varphi,\,\psi:X\to\R$ are locally Lipschitz continuous, then
\begin{enumroman}
\item\label{gd1} $\varphi^\circ(u;\cdot)$ is positively homogeneous,
sub-additive and continuous for all $u\in X$;
\item\label{gd2} $\varphi^\circ(u;-v)=(-\varphi)^\circ(u;v)$ for all $u,v\in X$;
\item\label{gd3} if $\varphi\in C^1(X)$, then $\varphi^\circ(u;v)=\langle\varphi'(u),v\rangle$ for all $u,v\in X$;
\item\label{gd4} $(\varphi+\psi)^\circ(u;v)\leq\varphi^\circ(u;v)+\psi^\circ(u;v)$ for all $u,v\in X$.
\end{enumroman}
\end{lem}

\begin{lem}\label{gg}
If $\varphi,\psi:X\to\R$ are locally Lipschitz continuous, then
\begin{enumroman}
\item\label{gg1} $\partial\varphi(u)$ is convex, closed and weakly$^*$ compact for all $u\in X$;
\item\label{gg2} the multifunction $\partial\varphi:X\to 2^{X^*}$ is upper semicontinuous with respect to the weak$^*$ topology on $X^*$;
\item\label{gg3} if $\varphi\in C^1(X)$, then $\partial\varphi(u)=\{\varphi'(u)\}$ for all $u\in X$;
\item\label{gg4} $\partial(\lambda\varphi)(u)=\lambda\partial\varphi(u)$ for all $\lambda\in\R$, $u\in X$;
\item\label{gg5} $\partial(\varphi+\psi)(u)\subseteq\partial\varphi(u)+\partial\psi(u)$ for all $u\in X$;
\item\label{gg6} for all $u,v\in X$ there exists $u^*\in\partial\varphi(u)$ such that  $\langle u^*,v\rangle=\varphi^\circ(u;v)$;
\item\label{gg7} if $g\in C^1(\R,X)$, then $\varphi\circ g:\R\to\R$ is locally Lipschitz, and $$\partial(\varphi\circ g)(t)\subseteq\Big\{\langle u^*,g'(t)\rangle:\,u^*\in\partial\varphi(g(t))\Big\}$$ for all $t\in\R$;
\item\label{gg8} if $u$ is a local minimizer (or maximizer) of $\varphi$, then $0\in\partial\varphi(u)$.
\end{enumroman}
\end{lem}

We also recall Lebourg's mean value theorem, see \cite[Proposition 1.3.14]{GP1}:

\begin{thm}\label{mvt}
If $\varphi:X\to\R$ is locally Lipschitz and $u,v\in X$, then there exist $\tau\in(0,1)$ and $w^*\in\partial\varphi(\tau u+(1-\tau)v)$ such that 
\[\varphi(v)-\varphi(u)=\langle w^*,v-u\rangle.\]
\end{thm}

By Lemma \ref{gg} \ref{gg1}, we may define for all $u\in X$
\begin{equation}\label{m}
m(u)=\min_{u^*\in\partial\varphi(u)}\|u^*\|_*,
\end{equation}
We say that $u\in X$ is a {\em critical point} of $\varphi$ if $m(u)=0$ (i.e. $0\in\partial\varphi(u)$). We denote by $K(\varphi)$ the set of critical points of $\varphi$ and, for any $c\in\R$, we set
\[K_c(\varphi)=\left\{u\in K(\varphi):\,\varphi(u)=c\right\}.\]
We say that $c\in\R$ is a {\em critical value} of $\varphi$ if $K_c(\varphi)\ne\emptyset$.

We say that $\varphi$ satisfies the {\em Cerami condition} (for short $(C)$) if, for every sequence $(u_n)_n$ in $X$ such that  $(\varphi(u_n))_n$ is bounded in $\R$ and $(1+\|u_n\|)m(u_n)\to 0$, there exists a convergent subsequence of $(u_n)_n$. Such a condition is commonly used in (both smooth and nonsmooth) critical point theory, as it fits better than the classical Palais-Smale condition with some special situations, especially in the non-coercive case. In particular, the $(C)$-condition appears in the following nonsmooth version of the {\em mountain pass theorem}, see \cite[Theorem 2.1.3]{GP1}.

\begin{thm}\label{mpt}
Let $\varphi:X\to\R$ be locally Lipschitz continuous and satisfy $(C)$. If $u_0,u_1\in X$, $\rho\in(0,\|u_1-u_0\|)$ are such that 
\[\inf_{\partial B_\rho(u_0)}\varphi=\eta_\rho>\max\{\varphi(u_0),\varphi(u_1)\},\]
and let
\[\Gamma=\left\{\gamma\in C([0,1],X):\,\gamma(i)=u_i,\, i=0,1\right\}, \quad c=\inf_{\gamma\in\Gamma}\max_{t\in[0,1]}\varphi(\gamma(t)),\]
then $c\geq\eta_\rho$ and $K_c(\varphi)\neq\emptyset$.
\end{thm}

In the proof of our main results we use the following technical lemma, which can be seen as a variant of the nonsmooth {\em implicit function theorem}, see \cite[Theorem 1.3.8]{GP1}:

\begin{lem}\label{ift}
Let $S\subseteq\partial B_1(0)$ be a nonempty set, let $\varphi:X\to\R$ be a locally Lipschitz continuous functional, $\mu<\inf_S \varphi$ be a real number, and suppose that the following conditions hold:
\begin{enumroman}
\item\label{ift1} $\displaystyle\lim_{t\to\infty}\varphi(tu)=-\infty$ for all $u\in S$;
\item\label{ift2} $\langle v^*,v\rangle<0$ for all $v\in\varphi^{-1}(\mu)$, $v^*\in\partial\varphi(v)$.
\end{enumroman}
Then there exists a continuous mapping $\tau:S\to(1,\infty)$ satisfying
\[\varphi(tu)
\begin{cases}
>\mu &\mbox{if $t<\tau(u)$}, \\
=\mu &\mbox{if $t=\tau(u)$}, \\
<\mu &\mbox{if $t>\tau(u)$} \\
\end{cases}\]
for all $u\in S$ and $t\ge 1$.
\end{lem}
\begin{proof}
Let us fix $u\in S$. By $(i)$ there exists $t>1$ such that  $\varphi(tu)=\mu$.

{\em Claim: $t$ is unique.} We argue by contradiction, assuming that there exist $t_1,t_2$ such that $1<t_1<t_2$ and $\varphi(t_i u)=\mu$ ($i=1,2$). By Lemma \ref{gg} \ref{gg6}, there exists $v^*\in\partial\varphi(t_1u)$ such that  $\varphi^\circ(t_1u;u)=\langle v^*,u\rangle$. Then, by \ref{ift2}, we have for some $\eps>0$
\[\varphi^\circ(t_1u;u)=\frac{1}{t_1}\langle v^*,t_1u\rangle<-\eps;\]
in particular
\[\limsup_{t\to t_1\above 0pt h\to 0^+}\frac{\varphi((t+h)u)-\varphi(tu)}{h}<-\eps.\]
So there exists $\eta>0$ such that  for all $t\in(t_1,t_1+\eta)$, $h\in(0,\eta]$
\[\varphi((t+h)u)<\varphi(tu)-h\eps.\]
Letting $t\to t_1$ we get $\varphi((t_1+h)u)<\mu$ for all $h\in(0,\eta]$, hence $\eta<t_2-t_1$. We define the closed set
\[I=\{t>t_1:\,\varphi(tu)=\mu\}.\]
Then $t_2\in I$ and $I\subseteq(t_1+\eta,\infty)$, so there exists $\bar t=\min I$. Clearly $\varphi(tu)<\mu$ for all $t\in(t_1,\bar t)$, while arguing as above we can find $\eta'>0$ such that  $\varphi(tu)<\mu$ for all $t\in(\bar t,\bar t+\eta)$. We see that $\bar t$ is a local maximizer of the locally Lipschitz continuous mapping $t\mapsto\varphi(tu)$. So, by Lemma \ref{gg} \ref{gg7}-\ref{gg8} there exists $w^*\in\partial\varphi(\bar tu)$ such that
\[\langle w^*,u\rangle=0.\]
This in turn implies $\langle w^*,\bar t u\rangle=0$ with $\bar tu\in\varphi^{-1}(\mu)$, against \ref{ift2}. This contradiction proves the Claim.

By the Claim, the function $\tau:S\to(1,\infty)$, which assigns to every $u\in S$ the only number $\tau(u)\in(1,\infty)$ satisfying $\varphi(\tau(u)u)=\mu$, is well-defined. Now we prove that $\tau$ is continuous. Let $u_0\in S$, $\eps\in(0,\tau(u_0)-1)$ and set $\tau_1=\tau(u_0)-\eps/2$, $\tau_2=\tau(u_0)+\eps/2$. Hence, by definition of $\tau$,
\[
\varphi(\tau_1 u_0)>\mu>\varphi(\tau_2 u_0).
\]
Being $\varphi$ continuous, we can find $\rho>0$ such that  $\overline B_\rho(\tau_1 u_0)\cup\overline B_\rho(\tau_2 u_0)\subset B_\eps(\tau(u_0)u_0)$ and $\varphi(u_1)>\mu>\varphi(u_2)$ for all $u_i\in \overline B_\rho(\tau_i u_0)$ ($i=1,2$). There exists $\delta>0$ such that  for all $u\in S\cap\overline B_\delta(u_0)$ we have $\tau_i u\in \overline B_\rho(\tau_i u_0)$ ($i=1,2$), so $\varphi(\tau_1u)>\mu>\varphi(\tau_2u)$. By the Claim, this in turn implies $\tau(u)\in(\tau_1,\tau_2)$, which rephrases as
\[|\tau(u)-\tau(u_0)|<\eps,\]
therefore $\tau$ is continuous.
\end{proof}

In view of the forthcoming results, we also need to recall some basic features of the {\em metric critical point theory} introduced by Degiovanni \& Marzocchi \cite{DM} (see also Degiovanni \cite{D} and Corvellec \cite{C3}). Assume that $(X,d)$ is a complete metric space and $\varphi:X\to\R$ is continuous. We define the {\em weak slope} of $\varphi$ at $u\in X$ as the supremum $|d\varphi|(u)$ of all numbers $\sigma\in\R$ for which there exist $\delta>0$ and a continuous mapping $H:[0,\delta]\times B_\delta(u)\to X$ such that
\[d(H(t,v),v)\le t, \quad \varphi(H(t,v))\le\varphi(v)-\sigma t \ \mbox{for all $(t,v)\in[0,\delta]\times B_\delta(u)$.}\]
Accordingly, we say that $u$ is a {\em (metric) critical point} of $\varphi$ if $|d\varphi|(u)=0$. We denote by $\tilde K(u)$ the set of such points and, for all $c\in\R$, we set
\[\tilde K_c(\varphi)=\left\{u\in\tilde K(\varphi):\,\varphi(u)=c\right\}.\]
It is well known that, if $\varphi$ is locally Lipschitz continuous, then $\tilde K_c(\varphi)\subseteq K_c(\varphi)$ for all $c\in\R$, while the reverse inclusion does not hold in general (see \cite[Theorem 3.9, Example 3.10]{D}). We say that $\varphi$ satisfies the {\em (metric) Palais-Smale condition} (shortly, $(\widetilde{PS})$) if, for any sequence $(u_n)_n$ in $X$ such that $(\varphi(u_n))_n$ is bounded and $|d\varphi|(u_n)\to 0$, $(u_n)_n$ admits a convergent subsequence.

Now we introduce a nonsmooth version of the {\em second deformation lemma} for locally Lipschitz continuous functionals, which was originally proved in the metric framework:

\begin{thm}\label{sdl}
If $\varphi:X\to\R$ is locally Lipschitz continuous and satisfies $(C)$, $a<b\leq \infty$ are such that  $K_a(\varphi)$ is a finite set, while $K_c(\varphi)=\emptyset$ for all $c\in(a,b)$, then there exists a continuous mapping $h:[0,1]\times(\overline\varphi^b\setminus K_b(\varphi))\to\overline(\overline\varphi^b\setminus K_b(\varphi))$ such that 
\begin{enumroman}
\item\label{sdl1} $h(0,u)=u$ and $h(1,u)\in\overline\varphi^a$ for all $u\in\overline\varphi^b\setminus K_b(\varphi)$;
\item\label{sdl2} $h(t,u)=u$ for all $(t,u)\in[0,1]\times\overline\varphi^a$;
\item\label{sdl3} $\varphi(h(t,u))\leq\varphi(u)$ for all $(t,u)\in[0,1]\times\overline\varphi^b\setminus K_b(\varphi)$.
\end{enumroman}
In particular, $\overline\varphi^a$ is a strong deformation retract of $(\overline\varphi^b\setminus K_b(\varphi))$
\end{thm}
\begin{proof}
We define on $X$ the so-called {\em Cerami metric} by setting for all $u,v\in X$
\[d_C(u,v)=\inf_{\gamma\in\Gamma_{u,v}}\int_0^1\frac{\|\gamma'(t)\|}{1+\|\gamma(t)\|}\,dt,\]
where $\Gamma_{u,v}$ denotes the set of all piecewise $C^1$ paths joining $u$ and $v$. The metric $d_C$ induces the same topology as $\|\cdot\|$, while $\widetilde{(PS)}$-sequences in $(X,d_C)$ coincide with $(C)$-sequences in $(X,\|\cdot\|)$ (see Corvellec \cite[Remark 4.2]{C3}). So, $\varphi$ is continuous and satisfies $\widetilde{(PS)}$ in $(X,d_C)$.

Moreover, by what observed above, $\tilde K_a(\varphi)$ is at most a finite set and $\tilde K_c(\varphi)=\emptyset$ for all $c\in(a,b)$. By \cite[Theorem 5.3]{C3}, there exists a deformation $h:[0,1]\times(\overline\varphi^b\setminus\tilde K_b(\varphi))\to(\overline\varphi^b\setminus\tilde K_b(\varphi))$ such that, for all $(t,u)\in [0,1]\times(\overline\varphi^b\setminus\tilde K_b(\varphi))$, the following conditions hold:
\begin{enumerate}[label=(\alph*)]
\item\label{msdl1} if $h(t,u)\neq u$, then $\varphi(h(t,u))<\varphi(u)$;
\item\label{msdl2} if $\varphi(u)\le a$, then $h(t,u)=u$;
\item\label{msdl3} if $\varphi(u)>a$, then $\varphi(h(1,u))=a$.
\end{enumerate}
We deduce that for all $(t,u)\in[0,1]\times(\overline\varphi^b\setminus K_b(\varphi))$ we have $h(t,u)\in\overline\varphi^b\setminus K_b(\varphi)$. Otherwise, there would exist $(t,u) \in [0,1]\times(\overline\varphi^b\setminus K_b(\varphi))$ with $h(t,u)\in K_b(\varphi)\setminus\tilde K_b(\varphi)$; in particular, we would have $\varphi(h(t,u))=b=\varphi(u)$, so that, by \ref{msdl1}, $h(t,u)=u$. Hence $u\in K_b(\varphi)$, a contradiction. In conclusion, we can restrict $h$ to a deformation of $\overline\varphi^b\setminus K_b(\varphi)$ and we easily see that it satisfies \ref{sdl1}-\ref{sdl3}.
\end{proof}

\section{Preliminaries II: nonsmooth Morse theory}\label{sec3}

In this section we discuss nonsmooth Morse theory. The main ideas are essentially contained in the paper of Corvellec \cite{C}, however our definitions slightly differ from those given therein, so we will provide some details. We refer to Motreanu, Motreanu $\&$ Papageorgiou \cite[Chapter 6]{MMP} for the properties of singular homology.

For all $B\subseteq A\subseteq X$, $H_k(A,B)$ denotes the $k$-th singular homology group of a topological pair $(A,B)$ (we choose $\R$ as the ring of coefficients, so $H_k(A,B)$ is a real linear space). Throughout this section we assume that $\varphi:X\to\R$ is locally Lipschitz continuous and satisfies $(C)$. Let $u\in K_c(\varphi)$ ($c\in\R$) be an {\em isolated} critical point, i.e., there exists a neighborhood  $U\subset X$ of $u$ such that  $K(\varphi)\cap U=\{u\}$. For all $k\in\N_0$ we define the $k$-{\em th critical group of $\varphi$ at $u$} as
\[C_k(\varphi,u)=H_k(\overline\varphi^c\cap U,\overline\varphi^c\cap U\setminus\{u\}).\]
Due to Theorem \ref{sdl} and the excision property of singular homology groups, $C_k(\varphi,u)$ is independent of the choice of $U$. We recall now a decomposition result for singular homology groups of sublevel sets of $\varphi$:

\begin{lem}\label{dec}
If $a<b\le\infty$, $c\in(a,b)$ is the only critical value of $\varphi$ in $[a,b]$, and $K_c(\varphi)$ is a finite set, then for all $k\in\N_0$
\[H_k(\overline\varphi^b,\overline\varphi^a)=\bigoplus_{u\in K_c(\varphi)}C_k(\varphi,u).\]
\end{lem}
\begin{proof}
Assume that $K_c(\varphi)=\{u_1,\ldots u_n\}$. We can find pairwise disjoint closed neigborhoods $U_1,\ldots U_n\subset X$ of $u_1,\ldots u_n$, respectively, and set $U=\bigcup_{i=1}^n U_i$. By Theorem \ref{sdl}, $\overline\varphi^a$ and $\overline\varphi^c$ are strong deformation retracts of $\overline\varphi^c\setminus\{u_1,\ldots u_n\}$ and $\overline\varphi^b$, respectively. By the excision property and \cite[Corollary 6.15, Proposition 6.18]{MMP}, we have
\begin{align*}
H_k(\overline\varphi^b,\overline\varphi^a) &= H_k(\overline\varphi^c,\overline\varphi^c\setminus\{u_1,\ldots u_n\}) = H_k(\overline\varphi^c\cap U,(\overline\varphi^c\cap U)\setminus\{u_1,\ldots u_n\}) \\
&= \bigoplus_{i=1}^n H_k(\overline\varphi^c\cap U_i,(\overline\varphi^c\cap U_i)\setminus\{u_i\}) = \bigoplus_{i=1}^n C_k(\varphi,u_i),
\end{align*}
which proves our assertion.
\end{proof}

Now assume that $\inf_{K(\varphi)}\varphi>c$ for some $c\in\R$. For all $k\in\N_0$ we define the $k$-{\em th critical group of $\varphi$ at infinity} as
\[C_k(\varphi,\infty)=H_k(X,\overline\varphi^c).\]
By Theorem \ref{sdl} and \cite[Corollary 6.15]{MMP}, $C_k(\varphi,\infty)$ is invariant with respect to $c<\inf_{K(\varphi)}\varphi$.

Henceforth, we will assume that $K(\varphi)$ is a {\em finite} set (in particular, all critical points are isolated). Now let $a<b$ be real numbers such that  $\varphi(u)\notin\{a,b\}$ for all $u\in K(\varphi)$. For all $k\in\N_0$ we define the $k$-{\em th Morse type number} and the $k$-{\em th Betti type number} of the interval $[a,b]$ as
\[M_k(a,b)=\sum_{u\in K(\varphi)\cap\varphi^{-1}([a,b])}{\rm dim}\,C_k(\varphi,u), \quad \beta_k(a,b)={\rm dim}\,H_k(\overline\varphi^b,\overline\varphi^a),\]
respectively. If $a<\varphi(u)<b$ for all $u\in K(\varphi)$, then again by Theorem \ref{sdl} and \cite[Corollary 6.15]{MMP} we have
\begin{equation}\label{betti}
\beta_k(a,b)={\rm dim}\,C_k(\varphi,\infty) \ \mbox{for all $k\in\N_0$.}
\end{equation}
Accordingly, we define two formal power series in the variable $t$, the {\em Morse polynomial} and the {\em Poincar\'e polynomial} as
\[M(a,b;t)=\sum_{k=0}^\infty M_k(a,b)t^k, \quad P(a,b;t)=\sum_{k=0}^\infty \beta_k(a,b)t^k,\]
respectively. The following identity, known as the nonsmooth {\em Morse relation}, will be the key tool in our study:

\begin{thm}\label{mr}
If $a<b$ are real numbers such that $\varphi(u)\in(a,b)$ for all $u\in K(\varphi)$, $M_k(a,b)<\infty$ for all $k\in\N_0$, and $M_k(a,b)=0$ for $k$ big enough, then there exists a polynomial with non-negative integer coefficients $Q(t)$ such that 
\[M(a,b;t)=P(a,b;t)+(1+t)Q(t) \ \mbox{for all $t\in\R$.}\]
\end{thm}
\begin{proof}
We first prove that for all $k\in\N_0$
\begin{equation}\label{morbet}
\beta_k(a,b)\le M_k(a,b).
\end{equation}
Indeed, assume that $c_1, \dots, c_m$, with $a<c_1<\ldots <c_m<b$, are the critical values of $\varphi$ in $[a,b]$. We fix $m+1$ real numbers $d_0, \dots,d_m$, such that  $a=d_0<c_1<d_1<c_2<\ldots<c_m<d_m=b$. By \cite[Lemma 6.56 (a)]{MMP} and Lemma \ref{dec}, we have for all $k\in\N_0$
\[{\rm dim}\, H_k(\overline\varphi^b,\overline\varphi^a) \le \sum_{j=1}^m {\rm dim}\,H_k(\overline\varphi^{d_j},\overline\varphi^{d_{j-1}}) = \sum_{j=1}^m\sum_{u\in K_{c_j}(\varphi)} {\rm dim}\,C_k(\varphi,u),\]
which yields \eqref{morbet}. In turn, \eqref{morbet} implies that $\beta_k(a,b)$ is finite and vanishes for $k$ big enough. Now, again by Lemma \ref{dec} and \cite[Lemma 6.56 (b)]{MMP} we have for all $t\in\R$
\[\sum_{k=0}^\infty M_k(a,b)t^k = \sum_{j=1}^m\sum_{k=0}^\infty \beta_k(d_j,d_{j
-1})t^k = \sum_{k=0}^\infty \beta_k(a,b)t^k+(1+t)Q(t)\]
for a convenient polynomial $Q$ with coefficients in $\N_0$. This concludes the proof.
\end{proof}

If in the Morse relation we choose $t=-1$, we obtain the nonsmooth {\em Poincar\'e-Hopf formula}:
\begin{equation}\label{ph}
\sum_{k=0}^\infty (-1)^k M_k(a,b)=\sum_{k=0}^\infty (-1)^k \beta_k(a,b).
\end{equation}

\section{Constant sign solutions}\label{sec4}

We provide problem \eqref{prob} with a variational formulation. Our study involves two function spaces, $W^{1,p}(\Omega)$  endowed with the norm $\|\cdot\|=(\|\nabla(\cdot)\|_p^p+\|\cdot\|_p^p)^{1/p}$ and $(C^1(\overline\Omega),\|\cdot\|_{C^1})$. The dual space of $W^{1,p}(\Omega)$ is $(W^*,\|\cdot\|_*)$. Moreover, for all $t\in[1,\infty]$ we denote by $\|\cdot\|_t$ the norm of $L^t(\Omega)$. Both $W^{1,p}(\Omega)$ and $C^1(\overline\Omega)$ are ordered Banach spaces with positive cones $W_+$ and $C_+$, respectively. We note that ${\rm int}(W_+)=\emptyset$, while
\[{\rm int}(C_+)=\left\{u\in C^1(\overline\Omega) \ : \ u(x)>0 \ \mbox{for all $x\in\overline\Omega$}\right\}.\]
By $A:W^{1,p}(\Omega)\to W^*$ we denote the $p$-Laplacian operator, i.e.
\[\langle A(u),v\rangle=\int_\Omega|\nabla u|^{p-2}\nabla u\cdot\nabla v\,dx \quad \mbox{for all $u,v\in W^{1,p}(\Omega)$.}\]

Now, we define some functionals on $W^{1,p}(\Omega)$. We denote by $s^\pm=\max\{\pm s,0\}$, and for all $(x,s)\in\Omega\times\R$ we set $j_\pm(x,s)=j(x,\pm s^\pm)$ (note that $j_\pm(x,\cdot)$ is locally Lipschitz continuous for a.a. $x\in\Omega$), and we fix $\eps>0$. For all $u\in W^{1,p}(\Omega)$ we define
\[\varphi(u)=\frac{\|\nabla u\|_p^p}{p}-\int_\Omega j(x,u)\,dx,\quad \varphi^\varepsilon_\pm(u)=\frac{\|\nabla u\|_p^p}{p}+\frac{\eps\|u^\mp\|_p^p}{p}-\int_\Omega j_\pm(x,u)\,dx\]
(note that $\varphi$ and $\varphi^\eps_\pm$ agree on $\pm W_+$). Moreover, for all $u\in L^t(\Omega)$ we set
\[N(u)=\left\{w\in L^{t'}(\Omega) \ : \ w\in\partial j(x,u) \ \mbox{a.e. in $\Omega$}\right\}\]
and
\[N_\pm(u)=\left\{w\in L^{t'}(\Omega) \ : \ w\in\partial j_{\pm}(x,u) \ \mbox{a.e. in $\Omega$}\right\}.\]
Finally, $m(u)$ and $m^\eps_\pm(u)$ are defined according to \eqref{m} for $\varphi$ and $\varphi^\varepsilon_\pm$, respectively.

Following a consolidated literature (see for instance Carl, Le \& Motreanu \cite{CLM}), we say that $u$ is a {\em (smooth weak) solution} of \eqref{prob} if $u\in C^1(\overline\Omega)$ and there exists $w\in N(u)$ such that
\[A(u)=w \ \mbox{in $W^*$.}\]

A fundamental property of the operator $A$ (see \cite[Proposition 2.72]{MMP}) is:
\begin{lem}\label{s+}
The mapping $A:W^{1,p}(\Omega)\to W^*$ is continuous and has the $(S)_+$ property, i.e., if $(u_n)_n$ is a sequence in $W^{1,p}(\Omega)$ such that  $u_n\rightharpoonup u$ in $W^{1,p}(\Omega)$ and
\[\limsup _{n\to \infty} \langle A(u_n),u_n-u\rangle\leq 0,\]
then $u_n\to u$ in $W^{1,p}(\Omega)$.
\end{lem}

The functional $\varphi$ is the nonsmooth energy of problem \eqref{prob}:
\begin{lem}\label{var1}
If hypotheses {\bf H} hold, then $\varphi:W^{1,p}(\Omega)\to\R$ is locally Lipschitz continuous and satisfies $(C)$. Moreover, if $u\in K(\varphi)$, then $u\in C^1(\overline\Omega)$ is a solution of \eqref{prob}.
\end{lem}
\begin{proof}
We assume $p<N$ (the argument for $p\geq N$ is easier). The functional $u\mapsto\frac{\|\nabla u\|_p^p}{p}$ is of class $C^1$ with derivative $A$. By {\bf H}\ref{h1}, the functional
\[u\mapsto\int_\Omega j(x,u)\,dx\]
is Lipschitz continuous on bounded sets of $L^r(\Omega)$ with generalized subdifferential contained in $N(\cdot)$ by the Aubin-Clarke theorem (see \cite[Theorem 1.3.10]{GP1}). By Lemma \ref{gg} \ref{gg5} and the continuous embedding $W^{1,p}(\Omega)\hookrightarrow L^r(\Omega)$, we deduce that $\varphi$ is locally Lipschitz continuous with
\begin{equation}\label{stella}
\partial\varphi(u)\subseteq A(u)-N(u) \ \mbox{for all $u\in W^{1,p}(\Omega)$.}
\end{equation}

Now we prove that $\varphi$ satisfies $(C)$. Let $(u_n)_n$ be a sequence in $W^{1,p}(\Omega)$ such that $(\varphi(u_n))_n$ is bounded in $\R$ and $(1+\|u_n\|)m(u_n)\to 0$. We first prove that $(u_n)_n$ is bounded in $W^{1,p}(\Omega)$. Indeed, by Lemma \ref{gg} \ref{gg1}, \eqref{m} and \eqref{stella}, for all $n\in\N$ there exists $w_n\in N(u_n)$ such that $m(u_n)=\|A(u_n)-w_n\|_*$. So, for all $v\in W^{1,p}(\Omega)$ we have
\begin{equation}\label{v3}
\left|\langle A(u_n),v\rangle-\int_\Omega w_n v\,dx\right|\le m(u_n)\|v\|.
\end{equation}
Taking $v=u_n$, we get
\begin{equation}\label{stellastella}
-\|\nabla u_n\|_p^p+\int_\Omega w_nu_n\,dx\le m(u_n)\|u_n\|=o(1).
\end{equation}
Moreover, since $(\varphi(u_n))_n$ is bounded, we have
\begin{equation}\label{v2}
\|\nabla u_n\|_p^p-\int_\Omega pj(x,u_n)\,dx\le c.
\end{equation}
Adding \eqref{stella} and \eqref{stellastella}, we obtain
\begin{equation}\label{v1}
\int_\Omega\left(w_n u_n-pj(x,u_n)\right)\,dx\le c.
\end{equation}
By {\bf H}\ref{h3} there exist $\beta,k>0$ such that
\begin{equation}\label{3stelle}
\xi s-pj(x,s)\geq\beta |s|^q \ \mbox{a.e. in $\Omega$, for all $|s|>k$, $\xi\in\partial j(x,s)$.}
\end{equation}
So, from \eqref{v1}, \eqref{3stelle} and {\bf H}$(i)$ we have
\[c \ge \int_{\{|u_n|>k\}}\beta |u_n|^q\,dx+\int_{\{|u_n|\leq k\}}(w_nu_n-pj(x,u_n))\,dx \ge \beta\|u_n\|_q^q-c.\]
Thus, $(u_n)_n$ is bounded in $L^q(\Omega)$. By {\bf H}\ref{h1} we know that $(r-p)\frac{N}{p}<r$, so we may assume $q<r<p^*$. We set
\[\tau=\frac{r-q}{r}\frac{Np}{Np-q(N-p)}\in(0,1),\]
hence we have $1/r=(1-\tau)/q+\tau/p^*$ and $\tau r<p$.
By the interpolation inequality (see Brezis \cite[Remark 2, p. 93]{B}) and the continuous embedding $W^{1,p}(\Omega)\hookrightarrow L^{p^*}(\Omega)$, we have
\[\|u_n\|_r \le \|u_n\|_{p^*}^\tau\|u_n\|_{q}^{1-\tau} \le c\|u_n\|^\tau.\]
By {\bf H}\ref{h1}, \eqref{v3}, H\"older inequality, and the inequality above, we obtain
\[\|u_n\|^p\le a_0(\|u_n\|_1+\|u_n\|_r^r) +\|u_n\|_p^p\le c(\|u_n\|+\|u_n\|^{\tau r}+\|u_n\|^{\tau p}).\]
Recalling that $\max\{1,\tau r,\tau p\}<p$, we finally deduce that $(u_n)_n$ is bounded in $W^{1,p}(\Omega)$. Passing to a subsequence, we may assume that $u_n\rightharpoonup u$ in $W^{1,p}(\Omega)$ and $u_n\to u$ in $L^\sigma(\Omega)$ for every $\sigma<p^\ast$. From \eqref{v3} we have for all $n\in\N$
\begin{align*}
\langle A(u_n),u_n-u\rangle &\le m(u_n)\|u_n-u\|+\int_\Omega |w_n(u_n-u)|\,dx \\
&\le m(u_n)\|u_n-u\|+\|w_n\|_{r'}\|u_n-u\|_r = o(1).
\end{align*}
By the $(S)_+$ property of $A$ in $W^{1,p}(\Omega)$ (see Lemma \ref{s+}), we have that $u_n\to u$ in $W^{1,p}(\Omega)$. So $(C)$ holds.

Now, let $u\in K(\varphi)$. Then, there exists $w\in N(u)$ such that $A(u)=w$ in $W^*$. By \cite[Theorem 1.5.5, Remark 1.5.9]{GP1} we have that $u\in L^\infty(\Omega)$. Nonlinear regularity theory (see Lieberman \cite[Theorem 2]{L}) then implies $u\in C^1(\overline\Omega)$. Thus, $u$ is a solution of \eqref{prob}.
\end{proof}

The aim of functionals $\varphi^\varepsilon_\pm$ is to select strictly positive or negative solutions:

\begin{lem}\label{var2}
If hypotheses {\bf H} hold, then $\varphi_\pm^\varepsilon:W^{1,p}(\Omega)\to\R$ are locally Lipschitz continuous and satisfy $(C)$. Moreover, if $u\in K(\varphi_\pm^\varepsilon)\setminus\{0\}$, then $u\in\pm{\rm int}(C_+)$ is a solution of \eqref{prob}.
\end{lem}
\begin{proof}
Let us consider $\varphi^\varepsilon_+$ (the argument for $\varphi^\eps_-$ is analogous). For all $(x,s)\in\Omega\times\R$ we have
\begin{equation}\label{gdj+}
\partial j_+(x,s)
\begin{cases}
=\{0\} &\mbox{if $s<0$}, \\
\subseteq\{\tau\xi \ : \ \xi\in\partial j(x,0),\ \tau\in[0,1]\} &\mbox{if $s=0$}, \\
=\partial j(x,s) &\mbox{if $s>0$}.
\end{cases}
\end{equation}
Arguing as in the proof of Lemma \ref{var1} we deduce that $\varphi^\varepsilon_+$ is locally Lipschitz and for all $u\in W^{1,p}(\Omega)$
\begin{equation}\label{u1}
\partial\varphi^\eps_+(u)\subseteq A(u)-\eps(u^-)^{p-1}-N_+(u).
\end{equation}

We prove that $\varphi^\varepsilon_+$ satisfies $(C)$. Let $(u_n)_n$ be a sequence in $W^{1,p}(\Omega)$ such that $(\varphi^\varepsilon_+(u_n))_n$ is bounded in $\R$ and $(1+\|u_n\|)m^\varepsilon_+(u_n)\to 0$. By Lemma \ref{gg} \ref{gg1}, for all $n\in\N$ there exists $w_n\in N_+(u_n)$ such that $m^\varepsilon_+(u_n)=\|A(u_n)-\eps(u_n^-)^{p-1}-w_n\|_*$. So, for all $v\in W^{1,p}(\Omega)$ we get 
\begin{equation}\label{v3+}
\left|\langle A(u_n),v\rangle-\int_\Omega\big(\eps(u_n^-)^{p-1}+w_n\big)v\,dx\right|\leq m^\varepsilon_+(u_n)\|v\|.
\end{equation}
Testing \eqref{v3+} with $v=u_n^+$, we have (note that $u_n^+u_n^-=0$ a.e in $\Omega$)
\[\|\nabla u_n^+\|_p^p-\int_\Omega w_nu_n^+\,dx=o(1),\]
while from the bound on $(\varphi(u_n))_n$ we have (note that $(u_n^+)^-=0$ a.e in $\Omega$)
\[\|\nabla u_n^+\|_p^p-\int_\Omega pj(x,u_n^+)\,dx\le c.\]
Applying {\bf H}\ref{h1}, \ref{h3} as in Lemma \ref{var1}, we see that $(u_n^+)_n$ is bounded in $L^q(\Omega)$. Testing \eqref{v3+} with $v=-u_n^-$, we have
\[\|\nabla u_n^-\|_p^p+\eps\|u_n^-\|_p^p=o(1),\]
hence $u_n^-\to 0$ in $W^{1,p}(\Omega)$. Reasoning as in Lemma \ref{var1}, we have $\|u_n\|_r\le c\|u_n\|^\tau$ for some $\tau\in[0,p/r)$. Testing \eqref{v3+} with $v=u_n$, we have
\begin{align*}
\|\nabla u_n\|_p^p &\le \eps\|u_n^-\|_p^p+\int_\Omega w_n u_n\,dx +o(1) \\
&\le c\big(\|u_n\|_q+\|u_n\|^{\tau r}\big)+o(1),
\end{align*}
which implies that $\|\nabla u_n\|_p\le c$. Since $\|\nabla(\cdot)\|_p+\|\cdot\|_q$ is an equivalent norm on $W^{1,p}(\Omega)$, we see that $(u_n)_n$ is bounded in $W^{1,p}(\Omega)$. Now, we argue as in Lemma \ref{var1} and find a convergent subsequence.

Finally, let $u\in K(\varphi_\pm^\varepsilon)\setminus\{0\}$. Since $0\in\partial\varphi^\varepsilon_+(u)$, by \eqref{u1} we find $w\in N_+(u)$ such that
\begin{equation}\label{ws+}
A(u) = \varepsilon(u^-)^{p-1}+w \ \mbox{in $W^*$.}
\end{equation}
Taking $v=u^-$, from \eqref{gdj+} we have
\[\|\nabla u^-\|_p^p+\varepsilon\|u^-\|_p^p=0,\]
hence $u\in W_+\setminus\{0\}$. Arguing as in Lemma \ref{var1}, we get $u\in C_+\setminus\{0\}$. By {\bf H}\ref{h5} we can apply the nonlinear maximum principle of V\'azquez \cite[Theorem 5]{V} (see also Pucci \& Serrin \cite[Theorem 1.1.1]{PS}) and we obtain $u\in{\rm int}(C_+)$. In particular, $N_+(u)=N(u)$, so by \eqref{ws+} $u$ is a solution of \eqref{prob}.
\end{proof}

The next Lemma deals with the zero solution:

\begin{lem}\label{zero}
If hypotheses {\bf H} hold, then $0$ is a local minimizer of both $\varphi$ and $\varphi^\eps_\pm$. In particular, $0$ is a solution of \eqref{prob}.
\end{lem}
\begin{proof}
We deal with $\varphi$ (the argument for $\varphi^\eps_\pm$ is analogous). Let $\delta_0>0$ be as in hypothesis {\bf H}\ref{h4} and set
\[B_{\delta_0}^C(0)=\left\{u\in C^1(\overline\Omega):\,\|u\|_{C^1}<\delta_0\right\}.\]
In particular, for all $u\in B_{\delta_0}^C(0)\setminus\{0\}$ we have $\|u\|_\infty<\delta_0$, hence
\[\varphi(u) \geq -\int_\Omega j(x,u)\,dx \ge -\int_\Omega j(x,0)\,dx = \varphi(0).\]
So, $0$ is a local minimizer of the restriction of $\varphi$ to $C^1(\overline\Omega)$. Reasoning as in Iannizzotto $\&$ Papageorgiou \cite[Proposition 3]{IP}, we see that $0$ is a local minimizer in $W^{1,p}(\Omega)$ as well, hence by Lemma \ref{gg} \ref{gg8} a critical point of $\varphi$. By Lemma \ref{var1}, $0$ solves \eqref{prob}.
\end{proof}

Now we are ready to introduce our two solutions result:

\begin{thm}\label{two}
If hypotheses {\bf H} hold, then problem \eqref{prob} admits at least two non-zero solutions $u_+\in{\rm int}(C_+)$ and $u_-\in -{\rm int}(C_+)$.
\end{thm}
\begin{proof}
We focus on positive solutions, so we consider functional $\varphi_+^\varepsilon$. From Lemma \ref{zero} we know that $0$ is a local minimizer of $\varphi^\varepsilon_+$. If $0$ is not a {\em strict} local minimizer, then we easily find another critical point $u_+\in K(\varphi_+^\varepsilon)$, which by Lemma \ref{var2} turns out to be a positive solution of \eqref{prob}.

So, we assume that $0$ is a strict local minimizer of $\varphi_+^\varepsilon$. So, there exists $\rho>0$ such that $\varphi^\varepsilon_+(u)>\varphi^\varepsilon_+(0)$ for all $u\in\overline B_\rho(0)\setminus\{0\}$. We prove that in fact
\begin{equation}\label{tw1}
\inf_{\partial B_\rho(0)}\varphi^\varepsilon_+:=\eta_\rho>\varphi^\varepsilon_+(0).
\end{equation}
We argue by contradiction: assume that there exists a sequence $(u_n)_n$ in $\partial B_\rho(0)$ such that $\varphi^\varepsilon_+(u_n)\to\varphi^\varepsilon_+(0)$ as $n\to\infty$. Passing to a subsequence, we may assume that $u_n\rightharpoonup u$ in $W^{1,p}(\Omega)$ and $u_n\to u$ in $L^t(\Omega)$ for all $t\in[p,p^*)$. Clearly $u\in\overline B_\rho(0)$. It is easily seen that $\varphi^\varepsilon_+$ is sequentially weakly lower semicontinuous in $W^{1,p}(\Omega)$, hence
\[\varphi^\varepsilon_+(u)\leq\liminf_n \varphi^\eps_+(u_n)=\varphi^\varepsilon_+(0),\]
which in turn implies $u=0$. Thus, by the relations above, we have
\[\frac{\|\nabla u_n\|_p^p}{p}+\frac{\varepsilon\|u_n\|_p^p}{p}=\varphi^\varepsilon_+(u_n)+\frac{\varepsilon\|u_n^+\|_p^p}{p}+\int_\Omega j_+(x,u_n)\,dx=o(1).\]
So $u_n\to 0$ in $W^{1,p}(\Omega)$, against the assumption that $u_n\in\partial B_\rho(0)$ for all $n\in\N$.

Now, by hypothesis {\bf H}\ref{h2} there exists $k\in\R$ such that
\[k>|\Omega|^{-\frac{1}{p}}\max\{\rho,|\eta_\rho|^\frac{1}{p}\}, \quad j(x,k)>k^p \ \mbox{a.e. in $\Omega$}\]
(by $|\Omega|$ we denote the $N$-dimensional Lebesgue measure of $\Omega$). Setting $\overline u(x)=k$ for all $x\in\Omega$, we have $\overline u\in W^{1,p}(\Omega)$ and $\|\overline u\|>\rho$. On the other hand,
\[\varphi^\varepsilon_+(\overline u)=-\int_\Omega j_+(x,k)\,dx<\eta_\rho.\]
We apply Theorem \ref{mpt} with $u_0=0$, $u_1=\overline u$, finding that $c\geq\eta_\rho$ and there exists $u_+\in K_c(\varphi^\varepsilon_+)$. From \eqref{tw1} we know that $u_+\neq 0$. Then, by Lemma \ref{var2}, $u_+\in{\rm int}(C_+)$ is a solution of \eqref{prob}.

A similar argument, applied to $\varphi^\varepsilon_-$, leads to the existence of a solution $u_-\in -{\rm int}(C_+)$.
\end{proof}

\section{Critical groups and the third solution}\label{sec5}

In this section we aim at proving the existence of a fourth critical point of the nonsmooth energy functional $\varphi$. Avoiding trivial situations, we may assume that both $K(\varphi)$ and $K(\varphi^\varepsilon_\pm)$ are {\em finite} sets. In particular, then, every critical point of $\varphi$ or $\varphi^\varepsilon_\pm$ is isolated. We begin by computing the critical groups at infinity of $\varphi$ (our result is the nonsmooth extension of \cite[Proposition 6.64]{MMP}):

\begin{lem}\label{cginf1}
If hypotheses {\bf H} hold, then $C_k(\varphi,\infty)=0$ for all $k\in\N_0$.
\end{lem}
\begin{proof}
We prove that
\begin{equation}\label{cgi1}
\lim_{t\to \infty}\varphi(tu)=-\infty \ \mbox{for all $u\in\partial B_1(0)$.}
\end{equation}
Indeed, fix $u\in \partial B_1(0)$ and choose $\sigma>0$ such that $\sigma\|u\|_p^p>\frac{\|\nabla u\|_p^p}{p}$. By {\bf H}\ref{h2}, there exists $k>0$ such that $j(x,s)>\sigma|s|^p$ a.e. in $\Omega$ and for all $|s|>k$. We can find $t_0>1$ such that $t_0|u|>k$ on some subset of $\Omega$ with positive measure. So, for $t>t_0$ big enough, we have
\begin{align*}
\varphi(tu) &= \frac{t^p\|\nabla u\|_p^p}{p}-\int_{\{|u|>k/t\}}j(x,tu)\,dx-\int_{\{|u|\leq k/t\}}j(x,tu)\,dx \\
&\le \frac{t^p\|\nabla u\|_p^p}{p}-\sigma\int_{\{|u|>k/t\}}|tu|^p\,dx+c \\
&\le t^p\left(\frac{\|\nabla u\|_p^p}{p}-\sigma\|u\|_p^p\right)+\sigma k^p|\Omega|+c,
\end{align*}
and the last quantity tends to $-\infty$ as $t\to\infty$, which proves \eqref{cgi1}.

Now, we prove that there exists $\displaystyle\mu<\inf_{K(\varphi)\cup\overline B_1(0)}\varphi$ such that
\begin{equation}\label{cgi2}
\langle v^*,v\rangle<0 \ \mbox{for all $v\in\varphi^{-1}(\mu)$, $v^*\in\partial\varphi(v)$.}
\end{equation}
First, by {\bf H}\ref{h3}, we can find $\beta>0,\tilde k>0$ such that
\[\xi s-pj(x,s)>\beta|s|^q \ \mbox{a.e. in $\Omega$ for all $|s|>\tilde k$, $\xi\in\partial j(x,s)$.}\]
For all $v\in W^{1,p}(\Omega)$ and $v^*\in\partial\varphi(v)$ there exists $w\in N(v)$ such that $v^*=A(v)-w$ in $W^*$. By {\bf H}\ref{h1} we have 
\begin{align*}
\int_\Omega\left(pj(x,v)-wv\right)\,dx &\leq \int_{\{|v|\leq \tilde k\}}\left(pj(x,v)-wv\right)\,dx-\beta\int_{\{|v|>\tilde k\}}|v|^q\,dx \\
&\leq c-\beta\int_{\{|v|>\tilde k\}}|v|^q\,dx \le \alpha,
\end{align*}
for some $\alpha>0$ independent of $v$. Applying the above inequality we have
\[\langle v^*,v\rangle = \|\nabla v\|_p^p-\int_\Omega wv\,dx \leq p\varphi(v)+\alpha .\]
Now we choose
\begin{equation}\label{cgi3}
\mu<\min\left\{\inf_{K(\varphi)\cup\overline B_1(0)}\varphi,-\frac{\alpha}{p}\right\}.
\end{equation}
For all $v\in\varphi^{-1}(\mu)$, $v^*\in\partial\varphi(v)$ we immediately get \eqref{cgi2}. Now Lemma \ref{ift} ensures the existence of a continuous mapping $\tau:\partial B_1(0)\to(1,\infty)$  such that for all $u\in\partial B_1(0)$, $t\geq 1$
\[\varphi(tu) \begin{cases}
>\mu &\mbox{if $t<\tau(u)$}, \\
=\mu &\mbox{if $t=\tau(u)$}, \\
<\mu &\mbox{if $t>\tau(u)$}.
\end{cases}\]
Clearly, from the choice of $\mu$ we have
\[\overline\varphi^\mu=\left\{tu:\, u\in\partial B_1(0),t\geq\tau(u)\right\}.\]
We set
\[D=\left\{tu:\, u\in\partial B_1(0),t\geq 1\right\}\]
and define a continuous deformation $h:[0,1]\times D\to D$ by putting for all $(s,tu)\in[0,1]\times D$
\[h(s,tu)= \begin{cases}
(1-s)tu+s\tau(u)u &\mbox{if $t<\tau(u)$}, \\
tu &\mbox{if $t\geq\tau(u)$}.
\end{cases}\]
Then, for all $tu\in D$ we have $h(1,tu)\in\overline\varphi^\mu$. Moreover, we have $h(s,tu)=tu$ for all $s\in[0,1]$, $tu\in\overline\varphi^\mu$. Hence, $\overline\varphi^\mu$ is a strong deformation retract of $D$. Besides, we set for all $(s,tu)\in[0,1]\times D$
\[\tilde h(s,tu)=(1-s)tu+su,\]
and we see that $\partial B_1(0)$ is a strong deformation retract of $D$ by means of $\tilde h$, as well. Applying \cite[Corollary 6.15]{MMP} twice, we have for all $k\in\N_0$
\[H_k(W^{1,p}(\Omega),\overline\varphi^\mu)=H_k(W^{1,p}(\Omega),D)=H_k(W^{1,p}(\Omega),\partial B_1(0))=0,\]
the last equality coming from \cite[Propositions 6.24, 6.25]{MMP} (recall that the sphere $\partial B_1(0)$ is contractible in itself, as ${\rm dim}(W^{1,p}(\Omega))=\infty$). Finally, recalling \eqref{cgi3} and the definition of critical group at infinity is enough to deduce $C_k(\varphi,\infty)=0$ for all $k\in\N_0$, which concludes the proof.
\end{proof}

Similarly we compute the critical groups at infinity of $\varphi^\varepsilon_\pm$:

\begin{lem}\label{cginf2}
If hypotheses {\bf H} hold, then $C_k(\varphi^\varepsilon_\pm,\infty)=0$ for all $k\in\N_0$.
\end{lem}
\begin{proof}
We deal with $\varphi^\eps_+$ (the argument for $\varphi^\eps_-$ is similar). We set
\[S_+=\left\{u\in\partial B_1(0) :\,\esssup_\Omega u>0\right\}, \ B_+=\{tu :\,t\in[0,1],u\in S_+\}.\]
As in Lemma \ref{cginf1}, using {\bf H}\ref{h1} - \ref{h3} we prove that
\[\lim_{t\to \infty}\varphi^\eps_+(tu)=-\infty \ \mbox{for all $u\in S_+$,}\]
and find
\[\mu<\inf_{K(\varphi^\eps_+)\cup B_+}\varphi(u)\]
such that
\[\langle v^*,v\rangle<0 \ \mbox{for all $v\in(\varphi^\eps_+)^{-1}(\mu)$, $v^*\in\partial\varphi^\eps_+(v)$.}\]
Now, Lemma \ref{ift} ensures the existence of a continuous mapping $\tau_+:S_+\to(1,\infty)$  such that for all $u\in S_+$, $t\geq 1$
\[\varphi^\eps_+(tu) \begin{cases}
>\mu &\mbox{if $t<\tau_+(u)$}, \\
=\mu &\mbox{if $t=\tau_+(u)$}, \\
<\mu &\mbox{if $t>\tau_+(u)$}.
\end{cases}\]
In particular, we have
\[\overline{(\varphi^\eps_+)}^\mu=\left\{tu:\,u\in S_+,\,t\ge\tau_+(u)\right\}.\]
We set
\[D_+=\left\{tu:\,u\in S_+,\,t\ge 1\right\}\]
and for all $(s,tu)\in[0,1]\times D_+$ we set
\[h_+(s,tu)=\begin{cases}
(1-s)tu+s\tau_+(u)u &\mbox{if $t<\tau_+(u)$}, \\
tu &\mbox{if $t\geq\tau_+(u)$},
\end{cases} \quad 
\tilde h_+(s,tu)=(1-s)tu+su.\]
Therefore, we see that $\overline{(\varphi^\eps_+)}^\mu$ and $S_+$ are strong deformation retracts of $D_+$ by means of $h_+$ and $\tilde h_+$, respectively. So we have for all $k\in\N_0$
\begin{equation}\label{cgi4}
H_k(W^{1,p}(\Omega),\overline{(\varphi^\eps_+)}^\mu)=H_k(W^{1,p}(\Omega),D_+)=H_k(W^{1,p}(\Omega),S_+).
\end{equation}
We set $u_0(x)=|\Omega|^{-1/p}$ for all $x\in\Omega$, hence $u_0\in S_+$. Set for all $(s,u)\in[0,1]\times S_+$,
\[\hat h(s,u)=\frac{(1-s)u+su_0}{\|(1-s)u+su_0\|}\]
(note that $\hat h$ is well defined as $u$ is positive somewhere), so $\hat h:[0,1]\times S_+\to S_+$ is a continuous deformation and $\{u_0\}$ turns out to be a strong deformation retract of $S_+$. So, $S_+$ is contractible in itself and by \eqref{cgi4} and the definition of critical groups at infinity we have $C_k(\varphi^\eps_+,\infty)=0$ for all $k\in\N_0$.
\end{proof}

From Lemma \ref{zero} and Theorem \ref{two} we know that $\varphi$ has at least three critical points, namely $0$, $u_+$, and $u_-$. We will conclude the proof of Theorem \ref{three} by proving the existence of a fourth critical point. We argue by contradiction, so in the following we assume
\begin{equation}\label{abs}
K(\varphi)=\{0,u_+,u_-\}
\end{equation}
(in particular, all critical points of $\varphi$ are isolated). We aim at applying the nonsmooth Poincar\'e-Hopf formula \eqref{ph}, so we need to compute the critical groups of $\varphi$ at all of its critical points (also the critical groups of $\varphi^\eps_\pm$ will be involved in our argument). We begin by computing the critical groups at $0$:

\begin{lem}\label{cg0}
If hypotheses {\bf H} and \eqref{abs} hold, then $C_k(\varphi,0)=C_k(\varphi^\eps_\pm,0)=\delta_{k,0}\R$ for all $k\in\N_0$.
\end{lem}
\begin{proof}
from Lemma \ref{zero} and \eqref{abs} we easily deduce that $0$ is a strict local minimizer of both $\varphi$ and $\varphi_\pm^\eps$. So, the conclusion follows at once from \cite[Axiom 7, Remark 6.10]{MMP}.
\end{proof}

Computation of the critical groups at $u_\pm$ is a more delicate issue:

\begin{lem}\label{cgpm}
If hypotheses {\bf H} and \eqref{abs} hold, then $C_k(\varphi,u_\pm)=\delta_{k,1}\R$ for all $k\in\N_0$.
\end{lem}
\begin{proof}
We consider $u_+$ (the argument for $u_-$ is analogous). First we prove that
\begin{equation}\label{shift}
C_k(\varphi,u_+)=C_k(\varphi^\eps_+,u_+).
\end{equation}
For all $t\in[0,1]$ we set $\psi_t=(1-t)\varphi+t\varphi^\eps_+$. Clearly, $\psi_t:W^{1,p}(\Omega)\to\R$ is locally Lipschitz continuous and satisfies $(C)$ for all $t\in[0,1]$; moreover $\psi_0=\varphi$ and $\psi_1=\varphi_+^\eps$. We claim that $u_+\in K(\psi_t)$ for all $t\in[0,1]$. Indeed, setting $\tilde\varphi=\restr{\varphi}{C^1(\overline\Omega)}$ and $\tilde\psi_t=\restr{\psi_t}{C^1(\overline\Omega)}$, we easily see that $\tilde\varphi,\tilde\psi_t:C^1(\overline\Omega)\to\R$ are locally Lipschitz. Since $0\in\partial\varphi(u_+)$, from the continuous embedding $C^1(\overline\Omega)\hookrightarrow W^{1,p}(\Omega)$, we see that $0\in\partial\tilde\varphi(u_+)$, i.e., $u_+\in K(\tilde\varphi)$.

Now, we recall that $u_+\in{\rm int}(C_+)$ and we note that $\tilde\varphi=\tilde\psi_t$ in $C_+$, so for all $v\in C^1(\overline\Omega)$ we have
\begin{align*}
\tilde\psi_t^\circ(u_+;v) &= \limsup_{w\to u_+\above 0pt \tau\to 0^+}\frac{\tilde\psi_t(w+\tau v)-\tilde\psi_t(w)}{\tau} \\
&= \limsup_{w\to u_+\above 0pt \tau\to 0^+}\frac{\tilde\varphi(w+\tau v)-\tilde\varphi(w)}{\tau} = \tilde\varphi^\circ(u_+;v) \ge 0,
\end{align*}
hence $0\in \partial \tilde\psi_t(u_+)$, that is, $u_+\in K(\tilde\psi_t)$. For all $v\in W^{1,p}(\Omega)$ there exists a sequence $(v_n)$ in $C^1(\overline\Omega)$ such that $v_n\to v$ in $W^{1,p}(\Omega)$, so by Lemma \ref{gd}\ref{gd1} we have
\[\psi_t^\circ(u_+;v) = \lim_n \psi_t^\circ(u_+;v_n) \ge \liminf_n\tilde\psi_t^\circ(u_+;v_n) \ge 0,\]
hence $u_+\in K(\psi_t)$.

Now, we prove that $u_+$ is an isolated critical point of $\psi_t$, uniformly with respect to $t\in[0,1]$, arguing by contradiction. In fact, assume that there exist sequences $(t_n)_n$ in $[0,1]$ and $(u_n)_n$ in $W^{1,p}(\Omega)\setminus\{u_+\}$ such that  $u_n\in K(\psi_{t_n})$ for all $n\in\N$ and $u_n\to u_+$ in $W^{1,p}(\Omega)$ as $n\to\infty$. Reasoning as in Lemmas \ref{var1} and \ref{var2}, for all $n\in\N$ we can find $w_n\in N(u_n)$ and $w_{n,+}\in N_+(u_n)$ such that  $u_n$ is a weak solution of the auxiliary problem
\begin{equation}\label{auxn}
\begin{cases}
-\Delta_p u=t_n\eps (u_n^-)^{p-1}+(1-t_n)w_n+t_n w_{n,+} &\text{in $\Omega$}, \\
\displaystyle\frac{\partial u}{\partial\nu}=0 &\text{on $\partial\Omega$}.
\end{cases}
\end{equation}
By {\bf H}$(i)$, we have for all $n\in\N$
\[|t_n\eps (u_n^-)^{p-1}+(1-t_n)w_n+t_n w_{n,+}|\le c(1+|u_n|^{r-1}) \ \mbox{a.e. in $\Omega$,}\]
with a constant $c>0$ independent of $n\in\N$. By \cite[Theorem 1.5.5, Remark 1.5.9]{GP1}, the sequence $(u_n)_n$ turns out to be bounded in $L^\infty(\Omega)$, and by \cite[Theorem 2]{L} $(u_n)$ is bounded also in $C^{1,\gamma}(\overline\Omega)$ ($\gamma\in(0,1)$). By the compact embedding $C^{1,\gamma}(\overline\Omega)\hookrightarrow C^1(\overline\Omega)$, passing if necessary to a subsequence, we have $u_n\to u_+$ in $C^1(\overline\Omega)$. So, for $n\in\N$ big enough we have $u_n\in{\rm int}(C_+)$, in particular $u_n^-=0$ and $N_+(u_n)=N(u_n)$. By definition of $N(u_n)$ and convexity of the set $\partial j(x,u_n)$, we have
\[(1-t_n)w_n+t_n w_{n,+}\in N(u_n).\]
Thus the right-hand side of \eqref{auxn} is a selection of $\partial j(\cdot,u_n)$ a.e. in $\Omega$ and $u_n\in{\rm int}(C_+)\setminus\{u_+\}$ is a solution of \eqref{prob}, against our assumption.

Finally, we note that the mapping $t\mapsto\psi_t$ is continuous with respect to the norm $\|\cdot\|_{1,\infty}$ in a neighborhood of $u_+$. So, we can apply the homotopy invariance of critical groups for non-smooth functionals (reasoning as in Corvellec \& Hantoute \cite[Theorem 5.2]{CH}) and conclude that $C_k(\psi_t,u_+)$ is independent of $t$. In particular we have \eqref{shift}.

Now, let $a,b\in\R$ be such that  $b<\varphi^\eps_+(0)<a<\varphi^\eps_+(u_+)$ (recall from the proof of Theorem \ref{two} that $\varphi^\eps_+(0)<\varphi^\eps_+(u_+)$). Set $A=\overline{(\varphi^\eps_+)}^a$, $B=\overline{(\varphi^\eps_+)}^b$, so that $B\subset A$ and by \cite[Proposition 6.14]{MMP} the following long sequence is exact:
\[\cdots\rightarrow H_k(W^{1,p}(\Omega),B)\xrightarrow{j_*} H_k(W^{1,p}(\Omega),A)\xrightarrow{\partial_*} H_{k-1}(A,B)\xrightarrow{i_*} H_{k-1}(W^{1,p}(\Omega),B)\rightarrow\cdots\]
Here $j_*$, $i_*$ are the homomorphisms induced by the inclusion mappings $j:(W^{1,p}(\Omega),B)\to (W^{1,p}(\Omega),A)$, $i:(A,B)\to (W^{1,p}(\Omega),B)$, respectively, and $\partial_*=\ell_* \circ \partial_k$, where $\ell:(A,\emptyset)\to (A,B)$ is the inclusion map and $\partial_k:H_k(W^{1,p}(\Omega),A)\to H_{k-1}(A)$ is the boundary homomorphism, see \cite[Definition 6.9]{MMP}. By \eqref{abs} and Lemma \ref{var2}, we have $K(\varphi^\eps_+)=\{0,u_+\}$. So, Lemma \ref{cginf2} implies
\[H_k(W^{1,p}(\Omega),B)=C_k(\varphi^\eps_+,\infty)=0.\]
Moreover, Lemma \ref{dec} implies
\[H_k(W^{1,p}(\Omega),A)=C_k(\varphi^\eps_+,u_+), \quad H_{k-1}(A,B)=C_{k-1}(\varphi^\eps_+,0).\]
So, the sequence above rephrases as the shorter sequence
\[0\rightarrow C_k(\varphi^\eps_+,u_+)\xrightarrow{\partial_*} C_{k-1}(\varphi^\eps_+,0)\rightarrow 0,\]
i.e., $\partial_*$ is an isomorphism. Then, by \eqref{cg0}, then, we have
\[C_k(\varphi^\eps_+,u_+)=\delta_{k-1,0}\R=\delta_{k,1}\R,\]
hence by \eqref{shift} we get the conclusion.
\end{proof}

Finally, we conclude the proof our main result.
\vskip6pt
\noindent
{\bf Proof of Theorem \ref{three}.} From Lemma \ref{zero} and Theorem \ref{two} we already know that \eqref{prob} admits the solutions $0$, $u_+$, and $u_-$. We argue by contradiction, assuming \eqref{abs}. We use Lemmas \ref{cginf1}, \ref{cg0}, and \ref{cgpm} into \eqref{ph}, so we obtain
\[\sum_{k=0}^\infty(\delta_{k,0}+2\delta_{k,1})(-1)^k=0,\]
i.e., $-1=0$, a contradiction. Thus, \eqref{abs} cannot hold, and there exists $\tilde u\in K(\varphi)\setminus\{0,u_+,u_-\}$. By Lemma \ref{var1}, we see that $\tilde u\in C^1(\overline\Omega)$ and $\tilde u$ is a solution of \eqref{prob}, which concludes the proof. \qed

\begin{rem}\label{further}
In the smooth case, Morse theory can be used in order to achieve further information on the solution set of a boundary value problem, for instance to produce a fourth non-zero solution if the equation is semilinear ($p=2$) and the reaction term is asymptotically linear at infinity (see for instance Mugnai $\&$ Papageorgiou \cite[Section 6]{MP}). Nevertheless, such results typically require the use of the second-order derivative of the energy functional, a notion which seems to have no counterpart within the framework of locally Lipschitz functionals on a Banach space.
\end{rem}

\noindent
\textbf{Ackowledgement.} We would like to thank Prof.\ N.S. Papageorgiou and Prof.\ J.N. Corvellec for their useful suggestions. Part of the present work was accomplished while F.C. was at the University of Perugia with a grant from {\em Accademia Nazionale dei Lincei}, under the supervision of Prof.\ P. Pucci, and A.I. was visiting the University of Perugia, which he recalls with deep gratitude. D.M. is member of the {\em Gruppo Nazionale per l'Analisi Matematica, la Probabilit\`a e le loro Applicazioni} (GNAMPA) of the {\em Istituto Nazionale di Alta Matematica} (INdAM), and is supported by the GNAMPA Project {\sl Systems with irregular operators}.

\end{document}